\documentclass{amsart}

\usepackage{graphicx}

\DeclareMathOperator*{\esssup}{ess\,sup}

\newtheorem{theorem}{Theorem}[section]
\newtheorem{lemma}[theorem]{Lemma}
\newtheorem{corollary}[theorem]{Corollary}
\newtheorem{proposition}[theorem]{Proposition}

\theoremstyle{definition}

\newtheorem{example}[theorem]{Example}

\theoremstyle{remark}
\newtheorem{remark}[theorem]{Remark}

\numberwithin{equation}{section}

\begin{document}

\title{Fractional Calculus of Variations: a novel way to look at it}

\author{Rui A. C. Ferreira}
\address{Grupo F\'isica-Matem\'atica, Faculdade de Ci\^encias, Universidade de Lisboa, Av. Prof. Gama Pinto, 2, 1649-003 Lisboa, Portugal and Departamento de Ci\^encias e Tecnologia, Universidade Aberta, Lisboa 1250-052, Portugal.}
\email{raferreira@fc.ul.pt}
\thanks{The author was supported by the ``Funda\c{c}\~ao para a Ci\^encia e a Tecnologia (FCT)" through the program ``Stimulus of Scientific Employment, Individual Support--2017 Call" with reference CEECIND/00640/2017.}






\begin{abstract}
In this work we look at the original fractional calculus of variations problem in a somewhat different way. As a simple consequence, we show that a fractional generalization of a classical problem has a solution without any restrictions on the derivative-order $\alpha$.  
\end{abstract}

\maketitle

\section{Introduction}

Fractional calculus and the variational calculus are subjects that have attracted many scientists over time. Historically, it seems to have been F. Riewe who first linked these two worlds in his work \cite{Riewe}. There, it was \emph{born} the \emph{fractional Euler--Lagrange equation}.

The classical calculus of variations (meaning Lagrangians depending on integer order derivatives) is a subject of immense research since the times of Euler. In particular, a correct proof of the (nowadays known by) Euler--Lagrange equation was given by Lagrange himself years after Euler derived it in a more or less heuristic way.

The modern approach to derive the classical Euler--Lagrange equation consists in essence on the use of the integration by parts formula and the du Bois-Reymmond lemma (or sometimes called \emph{fundamental lemma of the calculus of variations}) (see e.g. \cite{Brunt}). This being said, it sounds like a very simple exercise to do it. Well, the real scenario is not that simple if you want to construct a rigorous proof of it. Many questions arise that need justification in the process of deriving the formula. For example, under which space of functions can one use the integration by parts formula?

In the \emph{fractional scenario}, as reading the work of F. Riewe, one immediately recognizes that readilly from the beginning the author doesn't specify what kind of function  is the Lagrangian (smoothness conditions or other sort of things) nor the space of functions in which he wants to find a solution to the minimization problem (cf. \cite[Section III-C]{Riewe}). We may say that the Euler--Lagrange equation therein presented was derived formally.

The pioneering work of Riewe has attracted the interest of other scientists; we quote some of them here and refer the reader also to the references therein \cite{Agrawal,Almeida,Atanackovic,Bourdin,Cresson,Lazo,Malinowska}. While reading some of the (many) papers written in the subject we were led to some questions. As a prototype example, consider the following minimization problem (cf. the fractional calculus definitions, if needed, in Section \ref{sect2})
\vskip -12pt
\begin{equation}
    \label{Prob000}
    \mathcal{L}(y)=\int_0^1({^c}D^\alpha_{0^+}y(t))^2dt\rightarrow \min,\quad y(0)=0,\ y(1)=1.
\end{equation}
\vskip -2pt \noindent
With $\alpha=1$ the Caputo derivative is just the classical derivative $y'$. In many books on the calculus of variations topic, that is the first example used to show the usefulness of the Euler--Lagrange equation. In fact it is not so hard to show that the straight line $y(t)=t$ solves the problem \eqref{Prob000}. Note, in particular, that this solution is $C^{\infty}[0,1]$.

The problem \eqref{Prob000} was presented (except for a constant factor) as an example in \cite{Agrawal} (note that, since $y(0)$=0, the Riemann--Liouville and the Caputo derivatives coincide). We will dwell into the details later in this work (cf. Example \ref{ex22}), however, let us say for the moment that the solution of the Euler--Lagrange equation presented in \cite{Agrawal} was defined only for $\alpha>\frac{1}{2}$. That is somewhat counter-intuitive because a restriction arises in such a \emph{simple} problem. Moreover, in the space of admissible functions considered by the author in \cite{Agrawal}, we could actually show that the Euler--Lagrange equation does not have a solution unless $\alpha=1$ (again, this will be shown in Example \ref{ex22}). Immediately we recognized that one problem was precisely determining a suitable space of admissible functions in which \eqref{Prob000} will have a solution. In passage, we refer that such a space may be found in the paper \cite{Bourdin}. However we were not glad with such resolution because the restriction $\alpha>\frac{1}{2}$ still applies. Therefore, we went for a deeper study of problems such as \eqref{Prob000} and what we came to was that, if we looked at the Lagrangian a bit differently from the usual, we could obtain solutions to the Euler--Lagrange equation defined in a more \emph{friendlier} and recognizable space of functions, namely, the space for which the function and its Caputo fractional derivative are continuous on the entire interval $[a,b]$. Note that, when $\alpha=1$, we are just saying that $y$ and $y'$ are continuous on $[a,b]$, which is perfectly recognizable.

So, we propose the simple problem of the calculus of variations depending on the Caputo left derivative to be defined by ($0<\alpha\leq 1$)
\begin{align}
    \mathcal{L}(y)=\frac{1}{\Gamma(\alpha)}&\int_a^b(b-t)^{\alpha-1}L(t,y(t),{^c}D^\alpha_{a^+}y(t))dt\rightarrow \min,\label{prob09}\\
    & y(a)=y_a,\quad y(b)=y_b.\nonumber
\end{align}
Precise assumptions on the Lagrangian $L$, space of functions and so on will be given in Section \ref{sect3}. Incredibly, when we look at \eqref{prob09}, is the thought that this formulation is actually natural. Indeed, in the classical case one calculates the integral of a function $L$ that depends on a derivative $y'$ and, in the fractional case, we calculate the fractional integral of a function $L$ that depends on a fractional derivative ${^c}D^\alpha_{a^+}y$ (see \eqref{def1} below).

\vspace*{-4pt}

\begin{remark}
After the first draft of this paper was written, the author became aware of the work \cite{Ali}. In that paper the authors consider an optimal control problem in which the Lagrangian is defined analogously to the one in \eqref{prob09}. However, they do not provide any justification to do it so, therefore, we believe that this work will be useful for researchers working in this area, in particular the explanations provided in Section \ref{sect4}.
\end{remark}

This paper is organized as follows: in Section \ref{sect2} we provide to the reader some insights to the fractional calculus theory. In Section \ref{sect3} we rigorously derive the first order necessary condition for our working problem \eqref{prob09}. Finally, in Section \ref{sect4} we present two examples that will hopefully clarify the usefulness of this work.

\vspace*{-3pt}
\section{Fractional calculus}\label{sect2}

\setcounter{section}{2} \setcounter{equation}{0} 

In this section we make a brief introduction to the concepts and results we will use in this work. For a thorough study of the subject we refer the reader to the monographs \cite{Kilbas,samko}.

For $\alpha>0$ and an interval $[a,b]$, the Riemann--Liouville left and right fractional integrals of a function $f$ are defined by 
\vskip -10pt
\begin{equation}\label{def1}
I_{a^+}^\alpha f(t)=\frac{1}{\Gamma(\alpha)}\int_a^t(t-s)^{\alpha-1}f(s)ds,\quad I_{b^-}^\alpha f(t)=\frac{1}{\Gamma(\alpha)}\int_t^b(s-t)^{\alpha-1}f(s)ds,
\end{equation}
\vskip -3pt \noindent
respectively. We also put $I_{a^+}^0f=I_{b^-}^0f=f$.

The Riemann--Liouville left and right fractional derivatives of a function $f$ are defined by 
\vskip - 12pt
$$D^\alpha_{a^+}f(t)=\frac{d}{dt}I_{a^+}^{1-\alpha} f(t),\quad D^\alpha_{b^-}f(t)=-\frac{d}{dt}I_{b^-}^{1-\alpha} f(t),$$
\vskip -3pt \noindent
respectively.

Finally, the Caputo left and right fractional derivatives of a function $f$ are defined by 
\vskip -10pt
\begin{equation*}
{^c}D^\alpha_{a^+}f(t)=\frac{d}{dt}I_{a^+}^{1-\alpha} (f(s)-f(a))(t),\quad {^c}D^\alpha_{b^-}f(t)=-\frac{d}{dt}I_{b^-}^{1-\alpha} (f(s)-f(b))(t),
\end{equation*}
\vskip -3pt \noindent
respectively.

\vspace*{-4pt}

\begin{remark} 
We immediately conclude that the left Riemann\\--Liouville derivative and the left Caputo derivative coincide if $f(a)=0$. Evidently the same holds for the right derivatives provided $f(b)=0$.
\end{remark}

\vspace*{-3pt}

The following formulas will be used repeatedly.

\vspace*{-5pt}

\begin{proposition}\cite[cf. Property 2.1 p.71]{Kilbas}\label{prop111}
Let $\alpha,\beta>0$. Then,
$$I_{a^+}^\alpha (s-a)^{\beta-1}(t)=\frac{\Gamma(\beta)}{\Gamma(\beta+\alpha)}(t-a)^{\beta+\alpha-1},$$
$$D_{a^+}^\alpha (s-a)^{\beta-1}(t)=\frac{\Gamma(\beta)}{\Gamma(\beta-\alpha)}(t-a)^{\beta-\alpha-1},$$
$$I_{b^-}^\alpha (b-s)^{\beta-1}(t)=\frac{\Gamma(\beta)}{\Gamma(\beta+\alpha)}(b-t)^{\beta+\alpha-1},$$
$$D_{b^-}^\alpha (b-s)^{\beta-1}(t)=\frac{\Gamma(\beta)}{\Gamma(\beta-\alpha)}(b-t)^{\beta-\alpha-1},$$
\vskip -3pt \noindent
when they are defined.
\end{proposition}

The following result may be found in \cite[Lemma 2.4 p. 74]{Kilbas}. 
As usual, we denote by $L_p[a,b]$ ($1\leq p\leq\infty$) the set of Lebesgue functions defined on $[a,b]$ for which $\|f\|_p<\infty$, where
\vskip -12pt
$$\|f\|_p=\left(\int_a^b|f(t)|^p dt\right)^{\frac{1}{p}},\quad 1\leq p<\infty,$$
\vskip -5pt \noindent
and
\vskip -12pt
$$\|f\|_{\infty}=\esssup_{t\in[a,b]}|f(t)|.$$

\begin{proposition}\label{prop1}
For $\alpha>0$ and $f\in L_1[a,b]$, the following equalities hold
\vskip -12pt
$$D^\alpha_{a^+}I^\alpha_{a^+}f(t)=f(t),\quad D^\alpha_{b^-}I^\alpha_{b^-}f(t)=f(t)\ \ \mbox{a.e.\ on} \ [a,b].$$
\end{proposition}

\vspace*{-5pt}

\begin{remark}\label{remrem}
It can be shown that, if $f$ is continuous, then \\${^c}D^\alpha_{a^+}I^\alpha_{a^+}f(t)=f(t)$ on $[a,b]$ (cf. \cite[(2.4.32) p. 95]{Kilbas}).
\end{remark}

\vspace*{-5pt}

\begin{proposition}(cf. \cite[(2.4.42) p. 96]{Kilbas})\label{prop69}
If $f, {^c}D^\alpha_{a^+}f\in C[a,b]$, then $I^\alpha_{a^+}{^c}D^\alpha_{a^+}f(t)=f(t)-f(a)$ on $[a,b]$.
\end{proposition}

To present the next results, we need to introduce two spaces of functions. Let $1\leq p\leq \infty$ and $\alpha>0$. We define 
\vskip -8pt
$$I_{a^+}^\alpha(L_p[a,b])=\{f:f=I_{a^+}^\alpha\phi,\ \phi\in L_p[a,b]\},$$
\vskip -4pt \noindent
and
\vskip -12pt
$$I_{b^-}^\alpha(L_p[a,b])=\{f:f=I_{b^-}^\alpha\phi,\ \phi\in L_p[a,b]\}.$$

vskip 2pt

The next formula is crucial to derive the Euler--Lagrange equation and may be consulted in \cite[Corollary 2 p. 46]{samko}.

\begin{proposition}[Fractional integration by parts]\label{intpartes}
Let $0<\alpha<1$. Suppose that $f\in I_{b^-}^\alpha(L_p[a,b])$ and $g\in I_{a^+}^\alpha(L_q[a,b])$ with $1\leq p\leq\infty$, $1\leq q\leq\infty$ are such that $p^{-1}+q^{-1}\leq 1+\alpha$ ($p>1$ and $q>1$ when $p^{-1}+q^{-1}= 1+\alpha$). Then,
\begin{equation}\label{by parts}
\int_a^bf(x)D_{a^+}^\alpha g(x)dx=\int_a^bg(x)D_{b^-}^\alpha f(x)dx.
\end{equation}
\end{proposition}

\begin{remark}\label{rem10}
Obviously formula \eqref{by parts} still holds for $\alpha=1$ when $g(a)=0=g(b)$.
\end{remark}

\section{Problem formulation and main results}\label{sect3} 

\setcounter{section}{3} \setcounter{equation}{0} 

For $0<\alpha\leq 1$ consider the minimization problem
\begin{align}
    \mathcal{L}(y)=\frac{1}{\Gamma(\alpha)}&\int_a^b(b-t)^{\alpha-1}L(t,y(t),{^c}D^\alpha_{a^+}y(t))dt\rightarrow \min,\label{Prob}\\
    & y(a)=y_a,\ y(b)=y_b.\label{Prob2}
\end{align} 
We assume $L(t,u,v):[a,b]\times\mathbb{R}^2\rightarrow\mathbb{R}$ to be a continuous function with its partial derivatives $L_u$ and $L_v$ continuous. We will consider solutions to \eqref{Prob}--\eqref{Prob2} in the space of functions $\mathcal{F}=\{f:[a,b]\rightarrow\mathbb{R}\ s.t.\ f,{^c}D^\alpha_{a^+}f(t)\in C[a,b]\}$. The space of \emph{variations} is defined by $\mathcal{V}=\{f\in \mathcal{F}: f(a)=0=f(b)\}$.

We start by proving the fractional analogue of the classical  basic lemma of the  calculus of variations (see e.g. \cite[Section 1.4]{Mah}).

\vspace*{-2pt}

\begin{lemma}[Fractional du Bois-Reymond lemma]\label{B-R lemma}
Suppose that $f\in C[a,b]$. Then, for $0<\alpha\leq 1$
\vskip -10pt
\begin{equation}\label{eq0}
    \frac{1}{\Gamma(\alpha)}\int_a^b(b-s)^{\alpha-1}f(s){^c}D_{a^+}^\alpha\eta(s)ds=0,\quad\forall\eta\in\mathcal{V},
\end{equation}
if, and only if, $f(t)=k$ on $[a,b]$ for some constant $k\in\mathbb{R}$.
\end{lemma}

\vspace*{-3pt}

\begin{proof}
If $f(t)=k$ on $[a,b]$, then it follows from Proposition \ref{prop69} that
\vskip -8pt
$$\frac{1}{\Gamma(\alpha)}\int_a^b(b-s)^{\alpha-1}kD_{a^+}^\alpha\eta(s)ds=k(\eta(b)-\eta(b))=0.$$

Now, suppose that \eqref{eq0} holds, with $f\in C[a,b]$. Define
$$\eta(t)=I^\alpha_{a^+}f(t)-k\frac{(t-a)^\alpha}{\Gamma(\alpha+1)},\quad k=I^\alpha_{a^+}f(b)\frac{\Gamma(\alpha+1)}{(b-a)^\alpha}.$$
Since $\eta(a)=0=\eta(b)$ and $\eta,{^c}D_{a^+}^\alpha\eta\in C[a,b]$, then $\eta\in \mathcal{V}$. Moreover,
\vskip -10pt
\begin{multline*}
0=\frac{1}{\Gamma(\alpha)}\int_a^b(b-s)^{\alpha-1}f(s){^c}D_{a^+}^\alpha\eta(s)ds=\frac{1}{\Gamma(\alpha)}\int_a^b(b-s)^{\alpha-1}[f(s)-k]^2ds\\+\frac{k}{\Gamma(\alpha)}\int_a^b(b-s)^{\alpha-1}{^c}D_{a^+}^\alpha\eta(s)ds=\frac{1}{\Gamma(\alpha)}\int_a^b(b-s)^{\alpha-1}[f(s)-k]^2ds.
\end{multline*}
Therefore, $f(t)=k$ on $[a,b]$ and the proof is done.
\end{proof}

\vspace*{-4pt}

\begin{remark}
An attempt to prove a similar result to Lemma \ref{B-R lemma} was done in \cite[Lemma 3.2]{Lazo}. However the proof has inconsistencies. The authors assumed that the variations $\eta$ are differentiable on $[a,b]$. Within the proof they claim that the function $g(t)=I_{a^+}^\alpha(f(s)-K)(t),\ K\in\mathbb{R}$,
is differentiable on $[a,b]$ for a function $f\in L_1[a,b]$ such that there is a number $\varepsilon\in(a,b]$ with $|f(t)|\leq c(t-a)^{\beta}$ for all $t\in[a,\varepsilon]$, where $c>0$ and $\beta>-\alpha$ are constants. Well, we just need to define $f=C$ with $C>K$. Then (cf. Proposition \ref{prop111}), $g(t)=(C-K)\frac{(t-a)^\alpha}{\Gamma(\alpha+1)}$, which is not differentiable at $t=a$.
\end{remark}

\vspace*{-6pt}

\begin{theorem}\label{mainresult}
If $y\in\mathcal{F}$ is a solution of the minimization problem \eqref{Prob}--\eqref{Prob2}, then $y$ satisfies the equation
\begin{equation}\label{inteq}
   (b-t)^{1-\alpha}I_{b^-}^\alpha g(t)+L_v(t,y(t),{^c}D^\alpha_{a^+}y(t))=k\in\mathbb{R},\quad t\in[a,b],
\end{equation}
where $g(s)=(b-s)^{\alpha-1} L_u(s,y(s),{^c}D^\alpha_{a^+}y(s))$.
\end{theorem}

\vspace*{-4pt}

\begin{proof}
Let $0<\alpha\leq 1$. Under our hypothesis on $L$, $L_u$, $L_v$ and taking into account the spaces $\mathcal{F}$ and $\mathcal{V}$ we may conclude that, for a solution $y\in\mathcal{F}$ of \eqref{Prob}--\eqref{Prob2} and any (fixed) variation $\eta\in\mathcal{V}$, the first variation $\frac{d}{d\varepsilon}\mathcal{L}(y+\varepsilon\eta)$ equals zero at $\varepsilon=0$. Standard calculations then show that (note that ${^c}D^\alpha_{a^+}\eta(s)=D^\alpha_{a^+}\eta(s)$ since $\eta(a)=0$),
\vskip-14pt
\begin{multline*}
   \frac{1}{\Gamma(\alpha)} \int_a^b(b-s)^{\alpha-1}[L_u(s,y(s),{^c}D^\alpha_{a^+}y(s))\eta(s)\\
   +L_v(s,y(s),{^c}D^\alpha_{a^+}y(s))D^\alpha_{a^+}\eta(s)]ds=0.
\end{multline*}
\vskip -2pt \noindent
Now, since $\eta\in I_{a^+}^\alpha(L_\infty[a,b])$ (by Proposition \ref{prop69}) and $I_{b^-}^\alpha g\in I_{b^-}^\alpha(L_1[a,b])$, we conclude by Proposition \ref{intpartes} (and Remark \ref{rem10} if $\alpha=1$) that,
\vskip -12pt
$$\frac{1}{\Gamma(\alpha)}\int_a^b[I_{b^-}^\alpha g(s)D^\alpha_{a^+}\eta(s)+(b-s)^{\alpha-1}L_v(s,y(s),{^c}D^\alpha_{a^+}y(s))D^\alpha_{a^+}\eta(s)]ds=0,$$
\vskip -2pt \noindent
which is equivalent to
\vskip -13pt
$$\frac{1}{\Gamma(\alpha)}\int_a^b(b-s)^{\alpha-1}[(b-s)^{1-\alpha}I_{b^-}^\alpha g(s)+L_v(s,y(s),{^c}D^\alpha_{a^+}y(s))]D^\alpha_{a^+}\eta(s)ds=0.$$
\vskip -3pt \noindent
Now, we define the function $f(s)=(b-s)^{1-\alpha}I_{b^-}^\alpha g(s)+L_v(s,y(s),{^c}D^\alpha_{a^+}y(s))$ and we will show that $f\in C[a,b]$. Then, the equality \eqref{inteq} follows from Lemma \ref{B-R lemma}.

\vskip 2pt

The function $L_v(s,y(s),{^c}D^\alpha_{a^+}y(s))$ is obviously continuous on $[a,b]$, so we are left to show that the function $(b-s)^{1-\alpha}I^\alpha_{b^-}g(s)$ is also continuous on the entire interval. We start to prove the assertion for $s=b$:
We denote by $h(t)=L_u(t,y(t),{^c}D^\alpha_{a^+}y(t))$ and $M=\max_{t\in[a,b]}|h(t)|$.
We have,
\vskip -14pt
\begin{multline*}
    |(b-s)^{1-\alpha}I^\alpha_{b^-}g(s)|=\left|(b-s)^{1-\alpha}\frac{1}{\Gamma(\alpha)}\int_s^b(t-s)^{\alpha-1}(b-t)^{\alpha-1} h(t)dt\right|\\
    \leq M\frac{(b-s)^{1-\alpha}}{\Gamma(\alpha)}\int_s^b(t-s)^{\alpha-1}(b-t)^{\alpha-1}dt
    =M\frac{(b-s)^{1-\alpha}}{\Gamma(2\alpha)}(b-s)^{2\alpha-1},\ s<b,\\
    =M\frac{(b-s)^{\alpha}}{\Gamma(2\alpha)},\quad s\leq b.
\end{multline*}
\vskip -4pt \noindent
We now consider $a\leq c<b$. We want to show that
\vskip - 10pt
\begin{equation}\label{cont}
\lim_{s\rightarrow c}(b-s)^{1-\alpha}I^\alpha_{b^-}g(s)=(b-c)^{1-\alpha}I^\alpha_{b^-}g(c).
\end{equation}
\vskip -4pt \noindent
We have,
\vskip-14pt
\begin{multline*}
    |(b-s)^{1-\alpha}I^\alpha_{b^-}g(s)-(b-c)^{1-\alpha}I^\alpha_{b^-}g(c)|\\
    =\frac{1}{\Gamma(\alpha)}\left|(b-s)^{1-\alpha}\int_s^b(t-s)^{\alpha-1}(b-t)^{\alpha-1} h(t)dt\right.\\
    \left.-\, (b-c)^{1-\alpha}\int_c^b(t-c)^{\alpha-1}(b-t)^{\alpha-1} h(t)dt\right|.\\
\end{multline*}
\vskip -6pt \noindent
Now we assume that $b>s>c$, being the proof of the case $s<c$ analogous (of course that when $c=a$, then $s>a$). We have,
\vskip -14pt
   \begin{multline*}
 \frac{1}{\Gamma(\alpha)}\left|(b-s)^{1-\alpha}\int_s^b(t-s)^{\alpha-1}(b-t)^{\alpha-1} h(t)dt\right.\\
  \left.-(b-c)^{1-\alpha}\int_c^b(t-c)^{\alpha-1}(b-t)^{\alpha-1} h(t)dt\right|\\
    =\frac{1}{\Gamma(\alpha)}\left|\int_s^b[(b-s)^{1-\alpha}(t-s)^{\alpha-1}-(b-c)^{1-\alpha}(t-c)^{\alpha-1}](b-t)^{\alpha-1} h(t)dt\right.\\
    \left.-(b-c)^{1-\alpha}\int_c^s(t-c)^{\alpha-1}(b-t)^{\alpha-1} h(t)dt\right|=A.
\end{multline*}
\vskip -4pt \noindent
It is not difficult to see that
\vskip -11pt
$$(b-s)^{1-\alpha}(t-s)^{\alpha-1}-(b-c)^{1-\alpha}(t-c)^{\alpha-1}\geq 0\Leftrightarrow t\leq b,$$
\vskip -3pt \noindent
therefore,
\vskip -12pt
\begin{multline*}
  A\leq \frac{M}{\Gamma(\alpha)}\left[\int_s^b[(b-s)^{1-\alpha}(t-s)^{\alpha-1}-(b-c)^{1-\alpha}(t-c)^{\alpha-1}](b-t)^{\alpha-1}dt\right.\\
    \left.+\, (b-c)^{1-\alpha}\int_c^s(t-c)^{\alpha-1}(b-t)^{\alpha-1}dt\right]\\
    =\frac{M}{\Gamma(\alpha)}\left[(b-s)^{1-\alpha}\int_s^b(t-s)^{\alpha-1}(b-t)^{\alpha-1}dt\right.\\
   \left. -\,(b-c)^{1-\alpha}\int_c^b(t-c)^{\alpha-1}(b-t)^{\alpha-1}dt\right.\\
    \left.+2\, (b-c)^{1-\alpha}\int_c^s(t-c)^{\alpha-1}(b-t)^{\alpha-1}dt\right]\\
    \leq\frac{M}{\Gamma(\alpha)}\left[\frac{\Gamma(\alpha)}{\Gamma(2\alpha)}(b-s)^{\alpha}-\frac{\Gamma(\alpha)}{\Gamma(2\alpha)}(b-c)^{\alpha}\right.\\
    \left.+\, 2(b-c)^{1-\alpha}(b-s)^{\alpha-1} \frac{(s-c)^\alpha}{\alpha}\right]
    \end{multline*}
    \vskip -3pt \noindent
    which proves \eqref{cont}.
\end{proof} 

\begin{corollary}[Euler--Lagrange equation]
Suppose that $0<\alpha<1$. Then, under the conditions of Theorem \ref{mainresult}, $y$ satisfies the following equation 
\vskip -10pt
\begin{equation}\label{E-L}
    (b-t)^{\alpha-1}L_u(t,y(t),{^c}D^\alpha_{a^+}y(t))+D^\alpha_{b^-}h(t)=0,\quad t\in[a,b),
\end{equation}
\vskip -1pt \noindent
where $h(t)=(b-t)^{\alpha-1}L_v(t,y(t),{^c}D^\alpha_{a^+}y(t))$.
\end{corollary}

\begin{proof}
Since $y$ satisfies \eqref{inteq}, we have that
\vskip -12pt
$$I_{b^-}^\alpha g(t)+(b-t)^{\alpha-1}L_v(t,y(t),{^c}D^\alpha_{a^+}y(t))=k(b-t)^{\alpha-1},\quad t\in[a,b).$$
\vskip -3pt \noindent
Since $g\in L_1[a,b]$ and is not defined only at $t=b$, and $D^\alpha_{b^-}(b-t)^{\alpha-1}=0$ 
(cf. \cite[(2.1.21) p. 71]{Kilbas}), we conclude, using Proposition \ref{prop1}, that \eqref{E-L} holds.
\end{proof}

\vspace*{-5pt}

\begin{remark}
Note that when $\alpha=1$, we get from Theorem \ref{mainresult} the classical Euler--Lagrange equation:
\vskip -12pt
$$L_u(t,y(t),y'(t))-\frac{d}{dt}L_v(t,y(t),y'(t))=0,\quad t\in[a,b].$$
\vskip -3pt \noindent
\end{remark}

\section{Examples}\label{sect4} 

\setcounter{section}{4} \setcounter{equation}{0} 

We start by solving the prototype example we mentioned before within the text.

\vspace*{-4pt}

\begin{example}
For $0<\alpha\leq 1$ consider the problem of finding $y\in\mathcal{F}$ such that
\vskip -13pt
\begin{equation}\label{ex1}
    \mathcal{L}(y)=\frac{1}{\Gamma(\alpha)}\int_0^1(1-t)^{\alpha-1}({^c}D^\alpha_{0^+}y(t))^2dt\rightarrow \min,\quad y(0)=0,\ y(1)=1.
\end{equation}
Before we start by solving this problem we would like to emphasize that it could not be solved, to the best of our knowledge, by any known result in the literature. The reason is that the function $L(t,u,v)=(1-t)^{\alpha-1}v^2,\ \alpha<1$ is not continuous at $t=1$.

By Theorem \ref{mainresult} we know that a candidate $y\in\mathcal{F}$ must satisfy the equation
\vskip -14pt
$$2{^c}D^\alpha_{0^+}y(t)=k,\quad k\in\mathbb{R}.$$
Applying the operator $I_{0^+}^\alpha$ to both sides of the previous equality we obtain (cf. Proposition \ref{prop111} and Proposition \ref{prop69})
$$y(t)=k_1t^{\alpha}+k_2,\quad k_1,\ k_2\in\mathbb{R}.$$
Now with the boundary conditions we determine $k_1,\ k_2$ to finally get $y(t)=t^\alpha$ (when $\alpha=1$ this function is the straight line, in accordance to the classical case). Note that, indeed, $y\in\mathcal{F}$ (but $y$ in not differentiable at $t=0$, assumption that the reader may find in many works on the subject). Moreover, there is no restriction in the parameter $\alpha$. Now, to show that $y(t)=t^\alpha$ actually solves \eqref{ex1}, we use the fact that the function $L(v)=v^2$ satisfies the following inequality:
\vskip -11pt
$$L(x)-L(y)\geq 2y(x-y),\quad\forall x,y\in\mathbb{R}.$$
Then, for $y(t)=t^\alpha$ and $x\in\mathcal{F}$, we get
\begin{align*}
& \mathcal{L}(x)-\mathcal{L}(y)=\frac{1}{\Gamma(\alpha)}\int_0^1(1-t)^{\alpha-1}[({^c}D^\alpha_{0^+}x(t))^2-({^c}D^\alpha_{0^+}y(t))^2]dt\\
& \geq\frac{1}{\Gamma(\alpha)}\int_0^1(1-t)^{\alpha-1}2({^c}D^\alpha_{a^+}y(t))[{^c}D^\alpha_{0^+}x(t)-{^c}D^\alpha_{0^+}y(t)]dt\\[4pt] 
& =2kI_{0^+}^\alpha{^c}D^\alpha_{0^+}(x-y)(t)
=2k[x(1)-y(1)-(x(0)-y(0))] 
=0,
\end{align*}
where we have used Proposition \ref{prop69}. In conclusion, $y(t)=t^\alpha$ is a minimum of $\mathcal{L}$ given \eqref{ex1} subject to $y(0)=0$ and $y(1)=1$.
\end{example}

\begin{example}\label{ex22}
For $0<\alpha\leq 1$ consider the problem of finding $y\in\mathcal{F}$ such that
\vskip -12pt
\begin{equation}\label{ex2}
    \mathcal{L}(y)=\int_0^1({^c}D^\alpha_{0^+}y(t))^2dt\rightarrow \min,\quad y(0)=0,\ y(1)=1.
\end{equation}
First note that
$$\int_0^1({^c}D^\alpha_{0^+}y(t))^2dt=\frac{1}{\Gamma(\alpha)}\int_0^1(1-t)^{\alpha-1}\Gamma(\alpha)(1-t)^{1-\alpha}({^c}D^\alpha_{0^+}y(t))^2dt.$$
Therefore, $L(t,u,v)=\Gamma(\alpha)(1-t)^{1-\alpha}v^2$ is continuous as well as $L_v$. Hence, Theorem \ref{mainresult} tells us that the candidates for solving \eqref{ex2} should be found among the solutions of the differential equation
$$\Gamma(\alpha)(1-t)^{1-\alpha}\cdot{^c}D^\alpha_{0^+}y(t)=k,\quad k\in\mathbb{R}.$$
Letting $t=1$ in the previous equality (and recalling that ${^c}D^\alpha_{0^+}y(t)$ is assumed to be continuous on $[0,1]$) we immediately conclude that $k=0$ (except if $\alpha=1$). Therefore, in virtue of the boundary conditions, $y=0$ on $[0,1]$. We may conclude that the problem \eqref{ex2} does not have solutions in the space $\mathcal{F}$ for $\alpha<1$.

We emphasize that the solution of the Euler--Lagrange equation found in \cite[formula (42)]{Agrawal} for this problem (actually the problem therein considered used the Riemann--Liouville fractional derivative; however since $y(0)=0$ it coincides with the Caputo one) apart from a constant factor was, for $\frac{1}{2}<\alpha\leq 1$,
\vskip -13pt
$$y(t)=\int_0^t(t-s)^{\alpha-1}(1-s)^{\alpha-1}ds,$$
\vskip -4pt \noindent
for which
\vskip - 11pt
$$D^\alpha_{0^+}y(t)=(1-t)^{\alpha-1},$$
is not continuous on $[0,1]$, though the author assumed this in \cite[Theorem 1]{Agrawal}. The interested reader may find in the work of Bourdin \cite{Bourdin} other spaces of functions, rather than $\mathcal{F}$, where \eqref{ex2} has a solution, though with the restriction $\frac{1}{2}<\alpha\leq 1$.
\end{example}

In a way, these two examples indicate that the fractional generalization of the classical problem given by
\begin{equation*}
    \mathcal{L}(y)=\int_0^1(y'(t))^2dt\rightarrow \min,\quad y(0)=0,\ y(1)=1 
\end{equation*}
should be the one in Example \ref{ex1} rather than the one in Example \ref{ex2}. 
Nevertheless, we think that only the possible applications of a concrete problem to some physical phenomena will ultimately indicate the right formulation.

\vspace*{-3pt}

\section*{Acknowledgements}

The author was supported by the ``Funda\c{c}\~ao para a Ci\^encia e a Tecnologia (FCT)" through the program ``Stimulus of Scientific Employment, Individual Support--2017 Call" with reference CEECIND/00640/2017.


\end{document}